\documentclass[aoas]{imsart}

\RequirePackage{amsthm,amsmath,amsfonts,amssymb}
\RequirePackage[authoryear]{natbib}
\RequirePackage[colorlinks,citecolor=blue,urlcolor=blue]{hyperref}
\RequirePackage{graphicx}

\usepackage{multicol,tikz-cd,graphicx, hyperref,xcolor,pgf-pie,placeins,wasysym,colortbl,mathtools,xspace,multirow,nicefrac}

\startlocaldefs
\theoremstyle{plain}

\newtheorem{theorem}{Theorem}[section]
\newtheorem{lemma}[theorem]{Lemma}
\newtheorem{proposition}[theorem]{Proposition}
\newtheorem{question}{Question}
\theoremstyle{remark}

\newtheorem{remark}[theorem]{Remark}
\newtheorem*{example}{Example}

\definecolor{carnationpink}{rgb}{1.0, 0.65, 0.79}
\definecolor{chromeyellow}{rgb}{1.0, 0.65, 0.0}
\definecolor{classicrose}{rgb}{0.98, 0.8, 0.91}
\definecolor{maize}{rgb}{0.98, 0.93, 0.37}
\newcommand{\hcell}{\cellcolor{chromeyellow}}

\newcommand{\E}{\mathbb{E}}
\renewcommand{\P}{\mathbb{P}}
\newcommand{\D}{\mathcal{D}}
\newcommand{\Redist}{{\sf Redist}\xspace}

\endlocaldefs

\begin{document}

\begin{frontmatter}
\title{Repetition effects in a Sequential Monte Carlo sampler}
%\title{A sample article title with some additional note\thanksref{t1}}
%\runtitle{A sample running head title}
%\thankstext{T1}{A sample additional note to the title.}

\begin{aug}
%%%%%%%%%%%%%%%%%%%%%%%%%%%%%%%%%%%%%%%%%%%%%%%
%% Only one address is permitted per author. %%
%% Only division, organization and e-mail is %%
%% included in the address.                  %%
%% Additional information can be included in %%
\author[A]{\fnms{Sarah}~\snm{Cannon}\ead[label=e1]{scannon@cmc.edu}\orcid{0000-0001-6510-4669}},
\author[B]{\fnms{Daryl}~\snm{DeFord}\ead[label=e2]{daryl.deford@wsu.edu}\orcid{0000-0003-2032-3168}}
\and
\author[C]{\fnms{Moon}~\snm{Duchin}\ead[label=e3]{mduchin@cornell.edu}\orcid{0000-0003-4498-4067}}
%%%%%%%%%%%%%%%%%%%%%%%%%%%%%%%%%%%%%%%%%%%%%%
%% Addresses                                %%
%%%%%%%%%%%%%%%%%%%%%%%%%%%%%%%%%%%%%%%%%%%%%%
\address[A]{Department of Mathematical Sciences,
Claremont McKenna College\printead[presep={,\ }]{e1}}

\address[B]{Department of Mathematics and Statistics,
Washington State University\printead[presep={,\ }]{e2}}

\address[C]{Department of Mathematics and Brooks School of Public Policy,
Cornell University\printead[presep={,\ }]{e3}}
\end{aug}

\begin{abstract}
We investigate the prevalence of sample repetition in a Sequential Monte Carlo (SMC) method recently introduced for political redistricting.
\end{abstract}

\begin{keyword}
\kwd{Sequential Monte Carlo}
\kwd{sampling}
\kwd{graph partitions}
\kwd{redistricting}
\end{keyword}

\end{frontmatter}

% \documentclass{article}

% \usepackage[margin=1in]{geometry}
% \usepackage{multicol,tikz-cd,graphicx, hyperref,amsmath,amssymb,amsthm,xcolor,pgf-pie,placeins,wasysym,colortbl,mathtools,xspace,multirow,nicefrac}
% \hypersetup{colorlinks=True}
% \usepackage[font=small,margin=0.5in]{caption}

% %% theorems
% \newtheorem{theorem}{Theorem}
% \newtheorem{lemma}[theorem]{Lemma}
% \newtheorem{proposition}[theorem]{Proposition}
% \newtheorem{observation}[theorem]{Observation}
% \newtheorem{remark}[theorem]{Remark}
% \newtheorem{example}[theorem]{Example}
% \newtheorem{question}{Question}

% \definecolor{carnationpink}{rgb}{1.0, 0.65, 0.79}
% \definecolor{chromeyellow}{rgb}{1.0, 0.65, 0.0}
% \definecolor{classicrose}{rgb}{0.98, 0.8, 0.91}
% \definecolor{maize}{rgb}{0.98, 0.93, 0.37}
% \newcommand{\hcell}{\cellcolor{chromeyellow}}

% \newcommand{\R}{\mathbb{R}}
% \newcommand{\E}{\mathbb{E}}
% \renewcommand{\P}{\mathbb{P}}
% \newcommand{\D}{\mathcal{D}}
% \newcommand{\Redist}{{\sf Redist}\xspace}

% \title{Repetition effects in a Sequential Monte Carlo sampler}
% \author{Sarah Cannon, Daryl DeFord, and Moon Duchin}
% \date{\today}

% \begin{document}

\maketitle

\section{Introduction}

In this note, we consider the structure of the  SMC (Sequential Monte Carlo) method developed by McCartan--Imai for sampling partitions of a fixed node-weighted graph $G$ into a given number $k$ of connected subgraphs with nearly equal total  weight \citep{mccartan2023sequential}.
The main application is to redistricting, where the weight is by population, so we will refer to the partition as a {\em plan} and its parts as {\em districts}.
SMC operates by fixing a sample size $S$ and a number of districts $k$, then creating a top generation of partial plans by marking off one connected subset of $G$ with roughly $1/k$ of the total weight in each plan (Figure~\ref{fig:simple}, left).  
In the next generation, $S$ agents sample with replacement from those partial plans according to a weighting function; continuing the progress in the partial plan, the agents then mark off a second subset of about the same size.  This process continues for $k$ generations until the entire graph is marked.  We will summarize the salient combinatorial features of this scheme in a {\em descendancy diagram} $D$ (Figure~\ref{fig:simple}, right), showing only the paths that connect from the bottom layer to the top.  The nodes of $D$ involved in these paths are called {\em active} or {\em activated}.  The final generation consists of $S$ complete partitions.  This is the method developed by the SMC authors to generate ensembles of random districting plans, like the four complete plans at the bottom of Figure~\ref{fig:simple}.

\begin{figure}[bht!]
    \centering
\begin{tikzpicture}
\node at (0,0) {\includegraphics[width=3in]{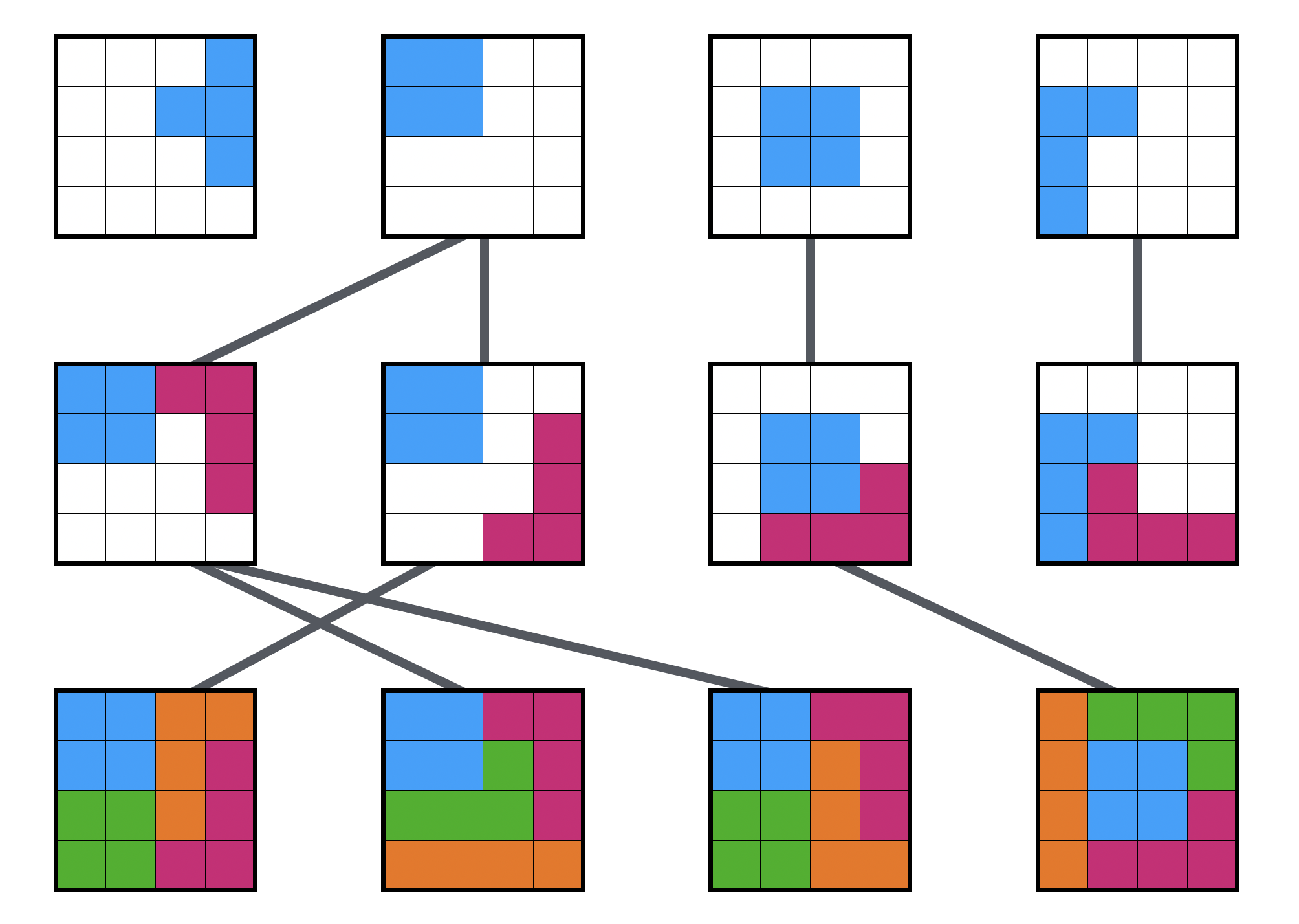}};
\node at (8,0) {\includegraphics[width=2in]{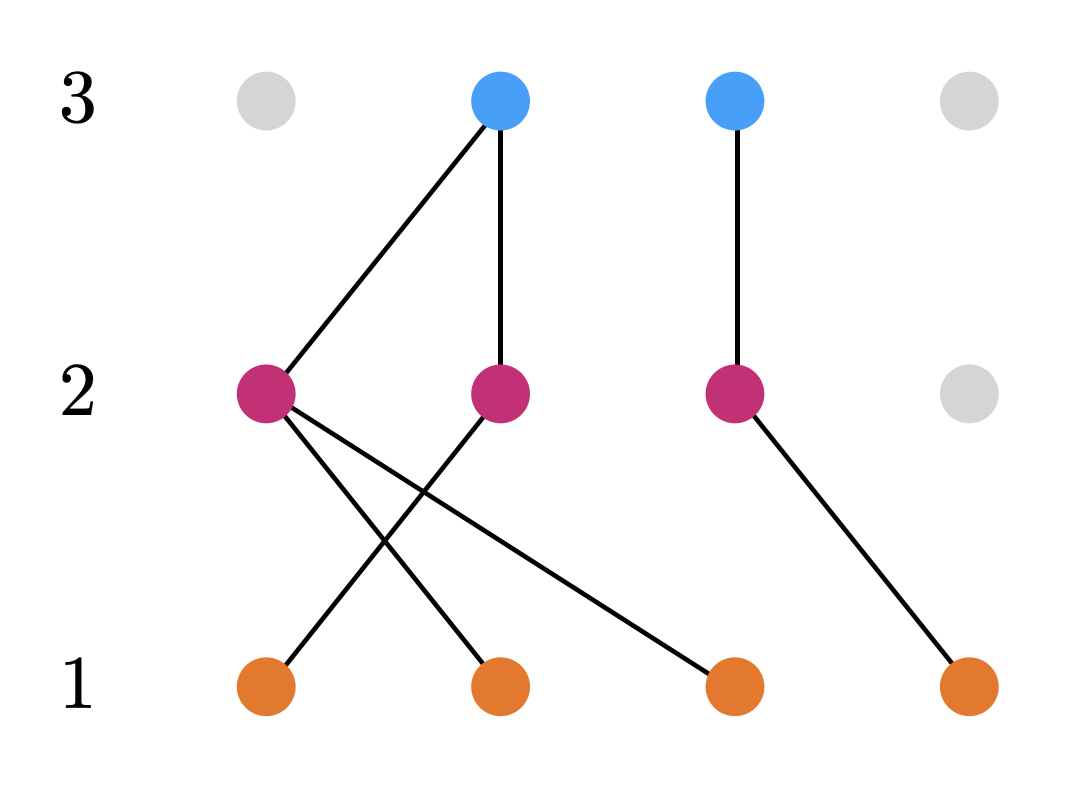}};
\end{tikzpicture}
    \caption{Simple example of partitioning a $4\times 4$ grid into four districts. The adjacency pattern of the grid is the graph $G$, the number of districts is $k=4$, and the size of the sample is $S=4$. At right, the process is abstracted into a {\em descendancy diagram}. The  district marked last (green) does not have a row of the diagram, because it is made up of area left over after the third district (orange) is marked.}
    \label{fig:simple}
\end{figure}

SMC samples face a certain amount of characteristic redundancy.  As the authors note, "because the SMC algorithm involves repeated resampling with replacement, for a finite number of samples it can suffer from particle system collapse \citep{LCL}, where many of the sampled plans share a small number of common districts (which originate as common ancestor particles in the SMC scheme)."  
In practical terms, this means that SMC samples can tend to have certain districts or sets of districts repeated many times across plans, like the blue square at the top left of the grid in Figure~\ref{fig:simple}, which appears in 3/4 of the final sample.
One aim of this note is to study this concentration of ancestry in combinatorial terms.

The second aim is to understand the tradeoffs faced by users in the choice of sampler. Citing McCartan--Imai again: "Like MCMC algorithms, the SMC algorithm generates samples which approximate the target distribution arbitrarily well as the sample size increases."  Our second set of questions explores the convergence guarantees and practical diagnostics available with SMC in relation to the production of various data artifacts used in redistricting analysis, like histograms and boxplots.  (See also \cite{cannon2022spanning} for more comparison of methods.)

Diagnosing strengths and weaknesses of the method has important real-world value.  
This is because SMC for redistricting, as implemented by the authors in the open-source package \Redist \citep{Redist}, is already in widespread use in courts of law. The repetition of districts has been flagged as a limitation that undermines statistical claims about the sample.  For instance, mathematician Kristopher Tapp produced an affidavit in  New York state Senate litigation in which he attempted a replication of another expert's SMC ensemble of 5000 maps, and found that a certain set of 31 districts (covering about half of the state) appeared identically in over 64\% of the sample.\footnote{Tapp further notes that while an ensemble of 5000 63-district maps can have up to 315,000 distinct districts, his replication ensemble had only 12,319, so that each district was repeated an average of 1360 times.  He calls this "a head-turning level of redundancy." \citep{Tapp}}  In  New Mexico state Senate litigation, a defense brief described the SMC method as being "plagued with duplicate simulations"; as a prophylactic measure to protect against this criticism, a defense expert cosmetically altered his SMC sample by perturbing the boundaries of districts so that he could  claim that no districts were duplicated.\footnote{"Dr.~Chen’s implementation of the MCMC version of an SMC algorithm [sic] did not
result in any duplicated maps. [Exh. D, Dep. ST 54:17–55:17 (falsely testifying that Dr.~Chen’s
simulations contain duplicates), 136:6–136:20 (correcting his mistaken testimony)]."  
The opposing expert Sean Trende was no more sophisticated, opining 
in deposition that "Duplicates happen
all the time... So it doesn't bother me, unless it gets extreme to where you end up having, like, 20 maps." \citep{NM-doc}}  It is clear that rigorous attention to the issues around duplication and the consequences for statistical interpretation are greatly needed in the field.

\subsection{Motivating questions}

To understand the SMC algorithm for $k$ districts, we will begin by studying {\em uniform descendancy diagrams} with levels (also called {\it layers} or {\it generations}) labeled $i=1$  to $k-1$ from bottom to top, each of width $S$, in which each node chooses a parent uniformly at random from the generation above. Nodes in the top layer (indexed $k-1$) represent partial plans with a single initial district marked, and those at level $i$ represent districting plans with $k-i$ districts marked. 
At level 1, $k-1$ districts are marked, which determines the $k$th and final district and amounts to specifying a complete plan.
A node in a descendancy diagram is {\it active} if it has a descendant in the bottom layer, and {\em surviving ancestors} are active nodes at the top level (so named because they have descendants in the final population).  Calculations using descendancy diagrams will omit all non-active nodes, because they have no role in the final sample constructed by SMC. 

\begin{figure}[htb!]
    \centering
\begin{tikzpicture}
\node at (0,0) {\includegraphics[width=2.8in]{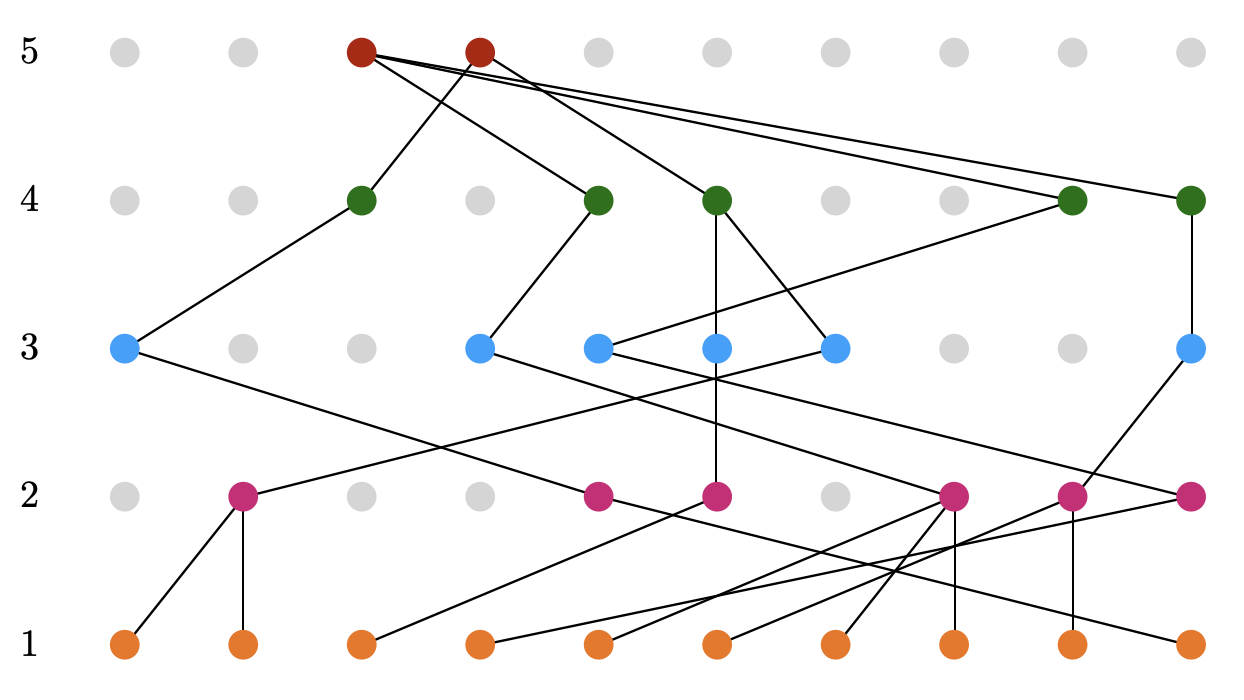}};
\node at (0,-2.5) {$S=10$, $k=6$};
\node at (8,0) {\includegraphics[width=3in]{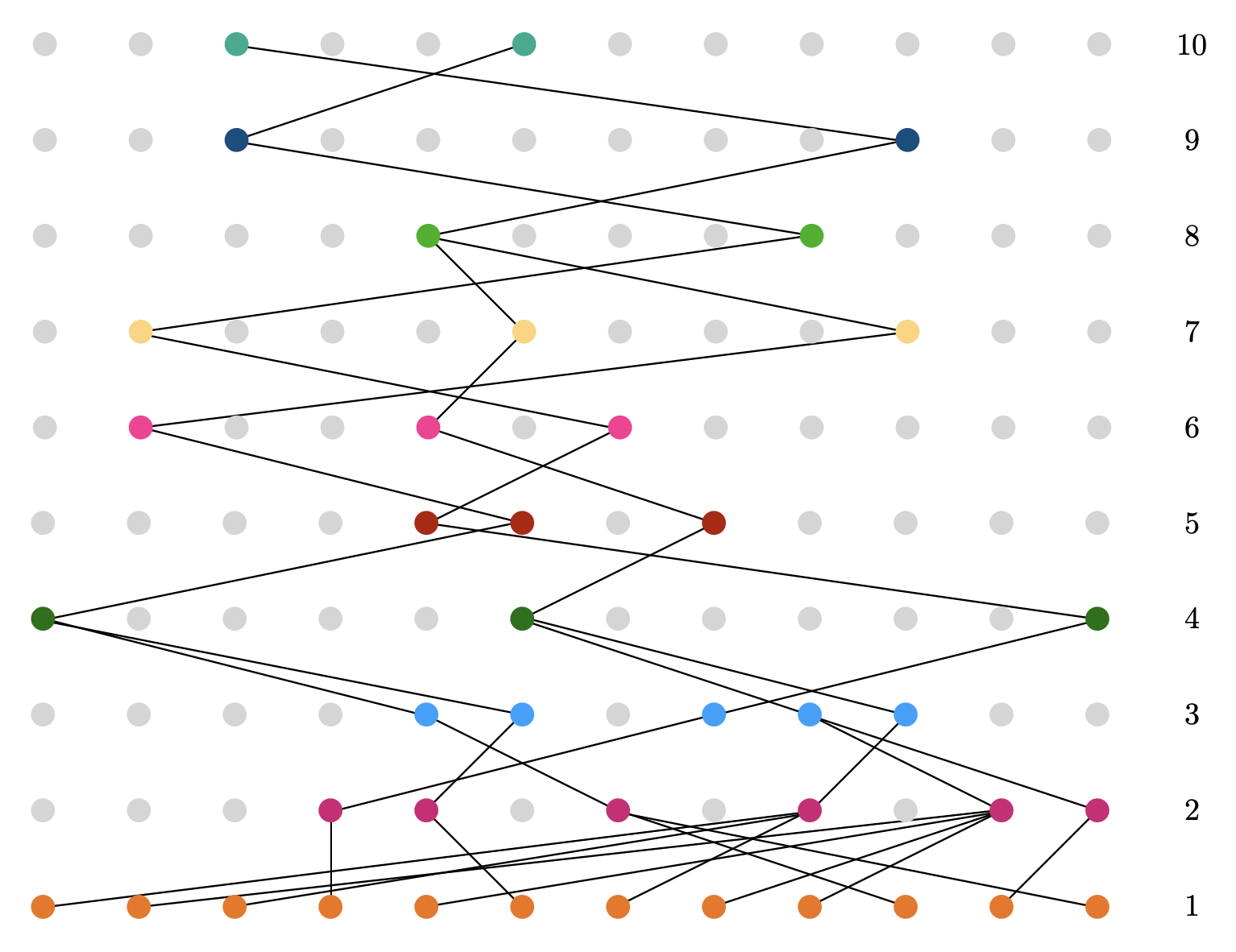}};
\node at (8,-3.8) {$S=12$, $k=11$};

\end{tikzpicture}
\caption{These two figures show structures we call {\em descendancy diagrams}.  The bottom row is labeled as generation 1 in each case, increasing in index with each layer until generation $k-1$ at the top.
Each of these two diagrams has $A(D)=2$, meaning that there are two top-level ancestors from which all  members of the bottom generation are descended.}
    \label{fig:ex1}
\end{figure}

If we write $D\in \D(S,k)$ for a specific diagram of this form, then let $A(D)$ be the number of surviving ancestors  in that diagram, and let $A(S,k)$ be the expected number of surviving ancestors $\E(A(D))$ as $D$ ranges over the uniform distribution on $\D(S,k)$. We consider the following questions.

\begin{figure}[htb!]
    \centering
    \includegraphics[width=3.9in]{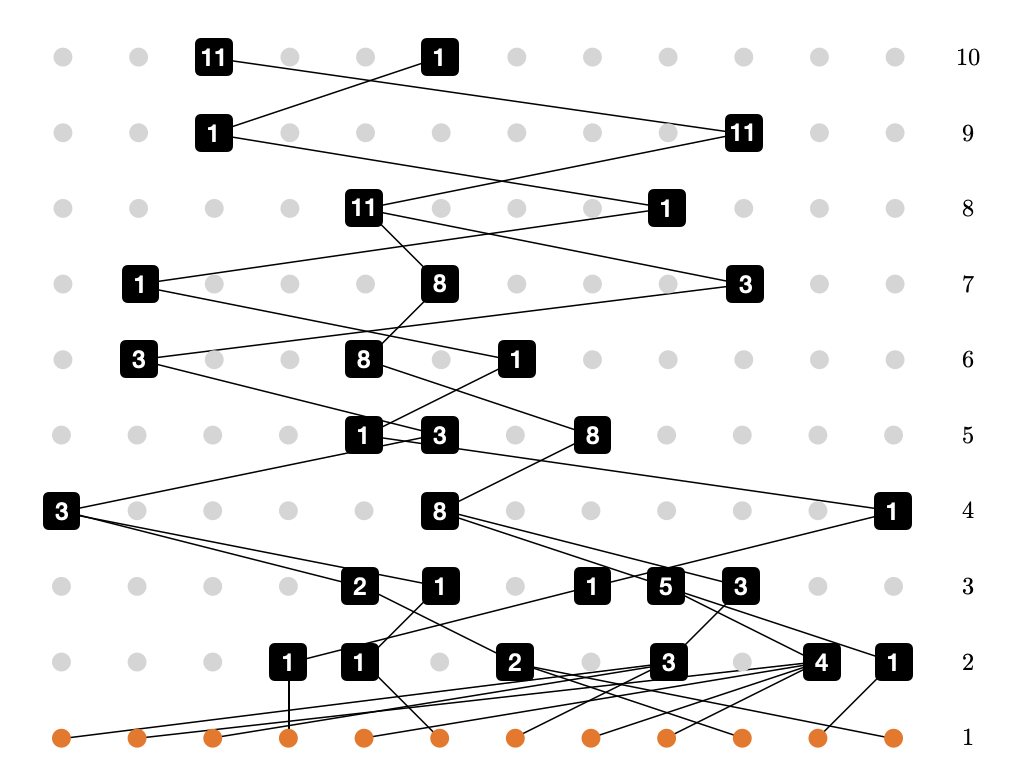}
    \caption{The $S=12,k=11$ example is repeated, now with the diagram nodes decorated by their number of final-generation descendants.  High numbers appearing on low levels are markers of extreme redundancy.}
    \label{fig:decorated}
\end{figure}

\begin{question}[Ancestor extinction]\label{Q:winners}
As a function of $S$ and $k$, what is the distribution of the number of surviving ancestors $A(D)$ in a uniform descendancy diagram?   Give bounds or asymptotics for the expectation $A(S,k)$.
\end{question}

\begin{question}[Extent of redundancy]\label{Q:Gdj}
Define $G(D,j)$ to be the number of final plans with at least $j$ districts in common. How is that $G$ distributed over $\D$ for each $j$?  
\end{question}

For instance, see Figure~\ref{fig:decorated}.  
 Seeing high numbers at low levels is an indicator of repetition.
In this case, there is one district that appears in 11 out of 12 final plans (value 11 appearing in top row).
Worse than that, it is part of a group of 3 districts that appear identically in those 11 plans (value 11 in row $8=k-3$).  Also, 
there is a group of 7 districts that appears in 8 out of 12 final plans (value 8 in row $4=k-7$).  
So for the sample of plans constructed according to this diagram $D$, we have $G(D,3)=11$ and $G(D,7)=8$.  

The questions so far concern the combinatorics of descendancy diagrams, 
but the combinatorics is only one source of redundancy.  It turns out to be compounded by several other factors.  One compounding factor is the non-uniformity of weights in the intermediate generations that are used when the each generation picks its parents.  A second factor is the graph being partitioned, which can start to have bottlenecks obstructing further district selection if the first few have been marked in an unlucky way.  

\begin{question}[Weighting and graph topology]
How do ancestry concentration and redundancy get more severe as the weighting factors deviate from uniform, and when realistic graphs are used as the basis for the partition?
\end{question}

Finally, we broaden the scope and consider the overall quality of the ensemble of $S$ plans that consists of members of the final generation.

\begin{question}[Convergence guarantees]
What are the convergence guarantees for the sampling distribution obtained from the weak SMC Central Limit Theorem \cite[Prop 4.2]{mccartan2023sequential}?  How do they limit the extent of redundancy?
\end{question}

\begin{question}[Summary statistics]
How do the convergence guarantees and diagnostics relate to the production of histograms, boxplots, and other percentile summary statistics?
\end{question}

There are many other elements of the SMC code as implemented in \Redist that pull results away from the simplicity of descendancy diagrams, besides those already mentioned (inter-generational weighting factors and bottlenecks in the graph topology).  
Another example is that SMC creates sequentially labeled districting plans but seeks to sample unlabeled plans; this is addressed with a corrective factor $\psi$ for which a second, auxiliary round of Monte Carlo estimation was described in the paper when sampling plans with $k>13$ districts.\footnote{Some district configurations have many more ways of being sequentially labeled than others; this can produce distortion factors of over a million in practice.  A corrective factor $\psi$ is described in \S4.4.2 of 
\citet{mccartan2023sequential}.}  Another example is a final reweighting step to align the sample more closely with the target distribution.  

The validation datasets employed in the article introducing SMC are graphs with 36 and 50 nodes and $k=3,4,6$ districts; it is not at all clear that some of the challenges faced by SMC will become visible at that scale.  The New York Senate case, in which the judge gave some credence to SMC ensemble made with default settings, had over 16,000 nodes and $k=63$ districts.  

Finally, it is very important to highlight that most of the SMC samples appearing in expert work, as well as the published work of the SMC authors, use small samples. 
 Typically, sample sizes of 5000 are presented based on a single run at $S=5000$ or by combining and subsampling multiple runs with $S=2500$.\footnote{The largest sample sizes we have found in the SMC 50-state data repository are at $S=30,000$ \citep{McCartan_ALARM}.  Our independent attempts to produce large runs  can reach $S=100,000$ on certain states, but it is difficult to get past that size using basic professional-level computer resources.  This is because the \Redist SMC code has significant memory overhead associated with storing partial plans.  Even on a medium-sized problem like redistricting Pennsylvania into $k=18$ districts at the precinct level, a run with $S=100,000$ consumes over 25 GB of RAM.}    
Thus for practical purposes, SMC sampling is currently engineered to be performed with multiple small runs, not with sample sizes in the millions.

Below, we will alternate between trying to isolate the effects of different features of the SMC sampler and simply reporting outcomes when the code is run.  This produces both a theoretical and a practical analysis.

\FloatBarrier

\subsection*{Acknowledgments}

We are grateful to Peter Rock for his excellent work conducting SMC experiments to support this project, building on his development of Python wrappers for the core \Redist functionality.  Replication repo  at \href{https://github.com/mggg/SMC-repetition}{\tt github.com/mggg/SMC-repetition}.
We  thank Cory McCartan for sharing his time to explain SMC for redistricting and the \Redist code, as well as for feedback and a correction on an earlier draft of this note.
We thank Peter Winkler and  Chris Hoffman for helpful conversations and pointers to the literature. 
This material is based upon work supported by the National Science Foundation under Grant No. DMS-1928930 and by the Alfred P. Sloan Foundation under grant G-2021-16778, while the authors were in residence at the Simons Laufer Mathematical Sciences Institute (formerly MSRI) during the Fall 2023 semester. SC is also supported in part by NSF CCF-2104795; MD by NSF DMS-2005512.

%%%%
\section{Structure of descendancy diagrams}

This section presents theoretical results, derived from the combinatorics of the descendancy diagrams, that allow us to bound the $A(S,k)$ values and begin to explain the expected behavior of collisions. We start by providing a Markov chain formulation for calculating exact values of $A(S,k)$ before providing more computationally efficient bounds on asymptotic behavior in terms of two recursive sequences. We conclude by examining the effect of weighting factors at each layer, proving that uniform choices minimize collisions and evaluating how much the non-uniformity explains the repetitions found in empirical data. 

\subsection{Setup}
We start with some simple observations about this model. When we take a uniform descendancy diagram with two layers (corresponding to an SMC process with three districts), Question~\ref{Q:winners} is a rephrasing of the classic birthday problem from probability. As $k$ grows, the generalized birthday problem also has a combinatorial interpretation as a sequential balls-in-bins model (see \cite{balls_and_bins}) or via random coagulations (see \cite{coagulations}).  Even in the language of ancestry, this has been studied in the context of genetic drift as the {\it Wright-Fisher Model} (see~\cite{durrett2002geneticdrift}). It is well known %(see, e.g., Chapter 4 of~\cite{coagulations}) 
that the probability that two individuals' lineages  remain distinct for at least $i$ levels is $(1-1/S)^i$ and the probability that  $\ell$ lineages remain (pairwise) distinct for at least $i$ levels is  $(1-1/S)^i (1-2/S)^i \ldots (1-(\ell-1)/S)^i$. After renormalizing and sending $S \rightarrow \infty$, the time of the first (pairwise) coalescence among $\ell$ distinct lineages approaches an exponential distribution with rate $\ell(\ell-1)/(2S)$. There is a rich literature considering variations of this model, using it to design Markov chains, and extending it to infinite $S$. 

Our setting differs slightly from the Wright-Fisher model by only considering lineages that extend to the bottom layer and discarding the others---in the language developed above, we only track active nodes. 

\begin{lemma}[One-step probabilities]
\label{lem:onestep}
    If a given generation $i$ has $1\leq t\leq S$ active nodes, then the expected number of  ancestors in the generation immediately above (generation $i+1$) is $S-S(1-\frac 1S)^t$. 
    The probability that there are exactly $v$ activated nodes in generation $i+1$ when there are $t$ activated nodes in generation $i$ is $P(v,t,S) = \binom{S}{v} \sum_{i=0}^v (-1)^{v-i}\binom{v}{i}\left(\frac{i}{S}\right)^t$. 
\end{lemma}

\begin{proof}
    In generation $i+1$, let $I_j$ be an indicator variable representing node $j$ being chosen at least once. For any individual $1\leq j\leq S$ we have $\P[I_j=1] = 1- (1-\frac 1S)^t$. Then the number of activated nodes is $\sum_{j=1}^S I_j$, and by linearity of expectation its expected value is $S-S(1-\frac 1S)^t$, as desired. 

    The second statement is a birthday problem variant.    For each set of $\binom{S}{v} $ parents the probability that all edges end up in that set is $\left(\frac{v}{S}\right)^t$ and applying inclusion/exclusion to account for versions that don't select every element in that set gives the desired result.
\end{proof}

With this lemma we can can compute the expected value across two or more generations exactly, but the formula does not give much insight, so we omit it.

We can reformulate the problem as a Markov chain on the states $1,2,3,\ldots, s$ representing the number of activated nodes at a given layer. This is an absorbing Markov chain with absorbing state 1, with transition probabilities given by $M_{i,j} = \begin{cases} P(i,j,s)& i \geq j\\0& i <j\end{cases}$, forming a lower-triangular transition matrix. The expectation we seek is $A(S,k)=\begin{bmatrix}0&0& \cdots & 1\end{bmatrix} M^{k-2} \begin{bsmallmatrix}1\\2\\ \vdots\\S\end{bsmallmatrix}$.
For example, when $S=3$, we have 
$M=\begin{bmatrix}
 1&0&0\\
1/3&2/3&0\\
1/9 & 6/9 & 2/9   
\end{bmatrix}$ and for diagrams with two layers ($k=3$) we have that the three nodes have an expected  $19/9=2.111...$ parents.

\subsection{Limiting behavior}

$A(S,k)$ is defined as the expected number of top nodes surviving to the bottom over diagrams $D\in \D(S,k)$.  
First, observe that for fixed $S$, we have $\lim\limits_{k\rightarrow\infty} A(S,k) = 1$. This follows directly from the Markov chain interpretation of the problem, because it is a non-increasing sequence of positive integers with a positive probability of strict decrease at each step while the value is greater than 1.

Given $S$, construct a sequence of coefficients as follows:  $a_{S,0}=1$; and $a_{S,i+1}=1-\left(1-\frac 1S\right)^{(a_{S,i}) S}$. 
Where $S$ is understood to be fixed we will write simply $a_0 =1$, 
$a_{i+1}=1-\left(1-\frac 1S\right)^{a_{i} S}$.
By Lemma~\ref{lem:onestep}, these approximate the share of active nodes at a given level of the diagram:  $a_{i} S \approx A(S,i)$.  We will get a rigorous upper bound below.

Note that as $S$ gets large, $\left(1-\frac 1S\right)^S$ rapidly converges to $1/e$ from below.
With this in mind, we can offer an second approximation with a  sequence given by $b_0=1$; and $b_{i+1}=1-\frac 1e^{b_i}$, which is more likely to have a useful  generating function, if an analytic description is desired. 
Table~\ref{tab:aibi} and Figure~\ref{fig:convergence rate} show the $b_i$ to be close to $a_{S,i}$ for large $S$---and show that $b_i\le A(S,i)\le a_{S,i}$ in all instances we investigated. 

\begin{table}[htb!]
    \centering
{\scriptsize    
    \begin{tabular}{|r||c|c|c|c|c|c|c|c|c|c|c|c|c|c|}
    \hline
$i$&0&1&2&3&4&5&6&7&8&9&10\\
\hline
\hline
$a_{10,i}$    & 1& 0.6513& 0.4965& 0.4073& 0.3490& 0.3077& 0.2769& 0.253& 0.234& 0.2185    &0.2056\\
\hline
$a_{100,i}$  &1& 0.6340& 0.4712& 0.3772& 0.3155& 0.2718& 0.2390& 0.2135&0.1056& 0.1931   & 0.1625\\
\hline
$a_{1000,i}$  &1& 0.6323& 0.4688& 0.3744& 0.3124& 0.2684& 0.2355& 0.2099& 0.1895& 0.1727 & 0.1587\\
\hline
$a_{5000,i}$  & 1& 0.6322& 0.4686& 0.3741& 0.3121& 0.2682& 0.2352& 0.2096& 0.1891& 0.1723 &0.1583\\
\hline
$b_i$         &1& 0.6321& 0.4685& 0.3741& 0.3121& 0.2681& 0.2352& 0.2095& 0.1890& 0.1723   & 0.1582\\
\hline\end{tabular}}
    \caption{Values of $a_{S,i}$ for $S\in \{10,100,1000,5000\}$ and $0 \leq i\leq 10$. As $S$ grows, the $a_{S,i}$ and $b_i$ get close. }
    \label{tab:aibi}
    
\end{table}

\FloatBarrier

\begin{lemma}
\label{lem:ak_limit}
For fixed $S>1$ and $a_i=a_{S,i}$, we have
   $\lim\limits_{i\rightarrow \infty} a_i = \frac1S$.
\end{lemma}
 \begin{proof} 
 First, we show by induction that $a_i \geq 1/S$ for all $i$. This is clearly true for $a_0 = 1$.  If $a_i \geq 1/S$, then $a_i S \geq 1$ and 
 \[
 a_{i+1} = 1-\left(1-\frac1S\right)^{a_i S} \geq 1-\left(1-\frac1S\right)^1 = \frac1S. 
 \]
We will also need the following fact, which follows from the inequality $ 1 + c < e^c$ for $c = 1/(S-1) \neq 0$: 
\begin{equation}\label{eqn:S-bound}
    -\left( 1-\frac1S \right) \ln \left(1-\frac1S\right)^S < 1.
\end{equation}

As we know $a_{i+1} = 1-\left(1-\frac1S\right)^{a_iS}$, we will focus on the function $f(x) = 1-\left(1-\frac1S\right)^{xS}$.  We begin by noting $f(\frac1S) = \frac1S$. We also see that 
\[
f'(x) = -\left(1-\frac1S\right)^{Sx} \ln \left(1-\frac1S\right)^S
\]
Equation~\eqref{eqn:S-bound} tells us that $f'(\frac1S) < 1$. 
For $x \geq \frac1S$, the values we are interested in, this slope is always positive (because  $\ln \left(1-\frac1S\right)^S$ is negative) and is strictly decreasing in $x$. This implies for $x > \frac1S$,  $f(x)$ is strictly bounded above by the line tangent to it at $1/S$: 
\[
f(x) < f\left(\frac1S\right) + f'\left(\frac1S\right) \left(x-\frac1S\right) = \frac1S + f'\left(\frac1S\right) \left(x-\frac1S\right)
\]

First, we will show the sequence of $a_i$'s is decreasing. Using $f'(\frac1S) < 1$, we see 
\[
a_{i+1} = f(a_i) < \frac1S + f'\left(\frac1S\right) \left(a_i-\frac1S\right) < \frac1S + a_i - \frac1S = a_i
\]
As the sequence of $a_i$'s is bounded below by $\frac1S$ and strictly decreasing, its limit must exist. 

Suppose, for the sake of contradiction, that the limit of the $a_i$'s is strictly greater than $1/S$, that is, it is $1/S + \alpha$ for some $\alpha > 0$. This means for all $\varepsilon > 0$, there is some sufficiently large $i$ such that $a_i < 1/S + \alpha + \varepsilon$. Choose $\varepsilon$ such that $0 < \varepsilon < \alpha (1-f'(\frac1S))/f'(\frac1S)$. This is possible to do because $\alpha > 0$ and $f'(\frac1S) < 1$. Note this choice of $\varepsilon$ means $f'(\frac1S) (\alpha + \varepsilon) < \alpha$. It follows that: 
\[
a_{i+1} = f(a_i) < \frac1S + f'\left(\frac1S\right) \left(a_i - \frac1S\right) < \frac1S + f'\left(\frac1S\right) (\alpha + \varepsilon) < \frac1S + \alpha.
\]
As this is a monotone decreasing sequence and we assumed its limit was $\frac1S + \alpha$, it is impossible to have $a_{i+1} < \frac1S + \alpha$, giving a contradiction.  Therefore it must be the case that the limit of this sequence is $\frac1S$, as claimed. 
\end{proof}

\begin{proposition}
\label{prop:exact_exp_limit} 
$A(S,k)  \leq   a_{k} S.$
\end{proposition}
\begin{proof} 
Let $X_i$ be a random variable denoting the number of active nodes at  level $i$ (those that have descendants in level $1$).  Thus $X_1\equiv S$. We are trying to get bounds on $\E[X_{k-1}]= A(S,k)$, the expected number of active nodes in the top level of a $k$-district descendancy diagram, which has $k-1$ levels.  

We will prove by induction that $\E[X_i] \le a_i S$, which suffices to prove the proposition.
When $i = 0$, we have  $\E[X_0] = a_0 S=S$, and the statement is true. 
Fix $i \geq 1$, and suppose $\E[X_i] \leq a_i S$. By Lemma~\ref{lem:onestep} and linearity of expectation, we have
$$\E[X_{i+1} \mid X_i] = S - S\left(1-\frac{1}{S}\right)^{X_i}.
$$

We will use the Law of Total Expectation ($\E[X_{i+1}] = \E[\E[X_{i+1} \mid X_i]]$) and  
Jensen's Inequality (for $c > 0$,  $\E[c^X] \geq c^{\E[X]}$).
We get
\begin{align*}
& \E[X_{i+1}] = \E[\E[X_{i+1} \mid X_i]] = \E\left[S - S\left(1-\frac{1}{S}\right)^{X_i}\right] = S-S\E\left[\left(1-\frac{1}{S}\right)^{X_i}\right]   
\\ & \hspace{12mm} \leq S-S\left(1-\frac{1}{S}\right)^{\E[X_i]}   \leq  
S-S\left(1-\frac{1}{S}\right)^{a_i S }  
%= S - S\left(1- a_{i+1} \right) 
=  a_{i+1} S. \quad \qedhere
\end{align*} 
\end{proof}

\begin{remark} %One could likely prove a lower bound using a Hoeffding-type inequality, but this would produce an estimate of the form $c^k a_k S \le A(S,k) \le a_k S$ for some $c<1.$ 
While we only show $a_{k} S$ is an upper bound, for large $k$ it is tight: in the limit as $k \rightarrow \infty$, it matches the trivial lower bound $\E[X_{k}] \geq 1$, which holds because there is at least one active node at each level.
The empirical results, such as those presented in Figure \ref{fig:convergence rate}, suggest the much stronger asymptotic 
$A(S,k)\sim a_{k} S$ as $S,k\to \infty$.  \end{remark}

\begin{figure}[bth!]
    \centering
\begin{tikzpicture}
\node at (0,0)    
{\includegraphics[height=1.6in]{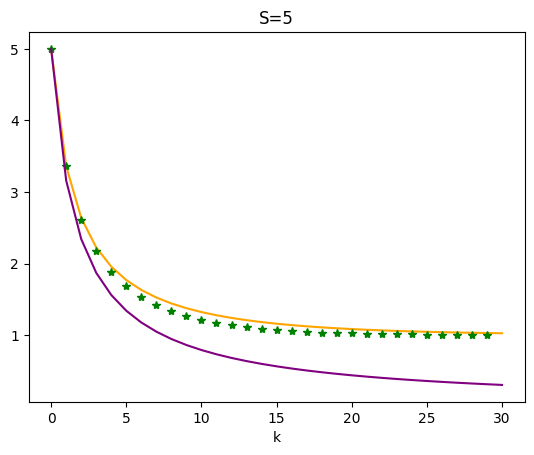}};
%\node at (0,1.8) {$S=5$};

\node at (-3,0) [rotate=90] {$A(S,k)$};

\node at (5,0) {\includegraphics[height=1.6in]{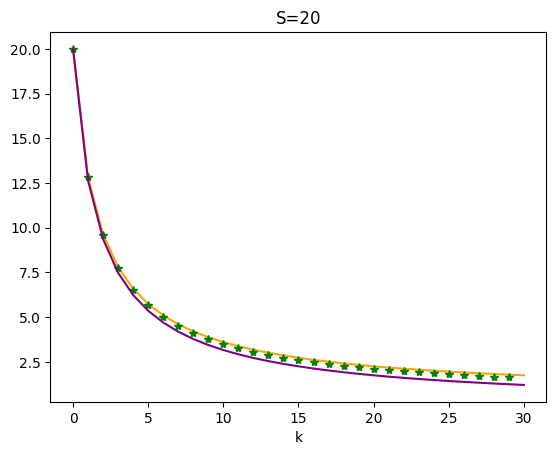}};
%\node at (5,1.8) {$S=10$};

\node at (10,0) {\includegraphics[height=1.6in]{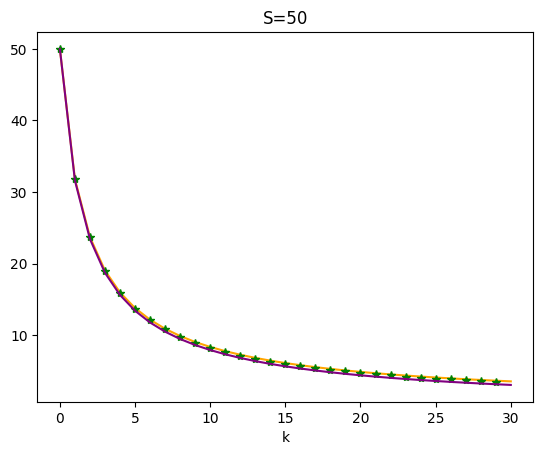}};
%\node at (10,1.8) {$S=20$};
\node at (12.2,-2.2) {$k$};
\end{tikzpicture}    
    \caption{If the distribution of weights is uniform, these plots show the expected number of surviving ancestors (districts drawn in the initial generation that appear in the final sample of plans) as $k$ grows, for $S=5,20,50$. The horizontal axis is $k$ in each plot and the vertical axis is the expected number of surviving ancestors. 
 Green stars are precise outputs from the Markov chain expression, compared to the $a_k S$ values in orange and the $b_k S$ values in purple (each interpolated by a curve).  In these small experiments, it is always true that $b_k S \le A(S,k)\le a_k S$, and that $a_k S\approx A(S,k)$ is a very good approximation.}
    \label{fig:convergence rate}
\end{figure}

Table~\ref{tab:k-sizes} shows that 
the predicted repetition (due to combinatorial collision expected under uniform weighting) is already pronounced, but empirical runs on real-world geography show far greater redundancy.  
We include $k=18$ (PA Congress 2010), $k=42$ (NM state House), 
$k=52$ (CA Congress 2020),
$k=63$ (NY State House), and
$k=203$ (PA State House).\footnote{The \Redist software has a built-in feature to warn users about "low plan diversity."  This warning message was triggered in 19/1000 PA Congressional runs, 534/1000 NM House runs, 939/1000 CA Congressional runs, 984/1000 NY House runs, and every PA House run. The message counsels users to "Consider weakening or removing constraints, or increasing the
population tolerance."  It is unclear how this kind of challenge was handled in the expert work cited in this note.
Furthermore, restricting to runs that do not trigger a warning message does not improve the repetition enormously.  In the same trials from Table~\ref{tab:k-sizes}, the max district repetition out of 5000 plans got as high as 2866 (PA-Cong), 4817 (NM-House), 4983 (CA-Cong), and 4622 (NY-House), even among runs with no low-diversity warning.}

\begin{table}[htb!]
    \centering
    \begin{tabular}{|c||c|c|c|c|c|}
    \hline
    & PA & NM &CA& NY & PA \\
    & Congress & House &Congress& House & House \\
    & $k=18$ & $k=42$& $k=52$& $k=63$ &$k=203$\\
    \hline
    \hline

$a_{5000,k}$  %&0.1011& 0.0455&0.0370&0.0308&0.0098 % old off-by-one values
&0.1066 & 0.0466 & 0.0377 &  0.0313 & 0.0099
\\
\hline 
$b_{k}$        %&0.1011& 0.0454&0.0369& 0.0307& 0.0097 % old off-by-one values
&0.1065 & 0.0464 & 0.0376 & 0.0312 & 0.0098  
\\
\hline
predicted average repetition %& 505.7 & 227.4 & 185.2 & 153.9 & 49.1 % old off-by-one values
 &\multirow{2}{*}{9.4} & \multirow{2}{*}{21.6} & \multirow{2}{*}{26.6} & \multirow{2}{*}{32.1} & \multirow{2}{*}{102.0}
 \\ 
 {\small $\nicefrac 1a \approx \nicefrac SA$} &&&&&\\
 \hline  \hline 
average district repetition  & \multirow{2}{*}{18.2} & \multirow{2}{*}{121.6} & \multirow{2}{*}{518.6} & \multirow{2}{*}{785.1} & \multirow{2}{*}{4905.0} \\
{\footnotesize (averaged over trials)} &&&&&\\
\hline 
max district repetition& \multirow{2}{*}{498.9} & \multirow{2}{*}{1896.9} & \multirow{2}{*}{3104.8} & \multirow{2}{*}{3350.8} & \multirow{2}{*}{4977.1} \\
{\footnotesize (averaged over trials)} &&&&&\\
\hline 
%empirical $A(D,k)$ & 276.5 & 43.8 & $*$ & 8.13 & $*$ \\
max district repetition& \multirow{2}{*}{4515} & \multirow{2}{*}{4951} & \multirow{2}{*}{5000} & \multirow{2}{*}{5000} & \multirow{2}{*}{5000}\\
{\footnotesize (max over trials)} &&&&&\\

\hline 
    \end{tabular}

\caption{For several realistic-sized problems, we consider the expected repetition of the initial districts that survive to the final sample.   
We report $1/a$ as predicted average repetition because it is a known bound that is quite tight already for small $S$. We conduct 1000 trials with $S=5000$ for each column; completing one trial with $k=203$ can take up to five days.  The fact that the observed average repetition is appreciably more severe 
than predicted points to the impact of other causes of redundancy, like non-uniform weights, graph bottlenecks, and final reweighting---these have a snowballing impact as the number of districts grows.  These will be explored below.}
    \label{tab:k-sizes}
\end{table}

\begin{remark}
The case of square diagrams is a natural one to consider.
Numerical results suggest that 
$A(S,S)$ limits to a constant slightly greater than 2.
Subsequently, once there are two active nodes in a population of $S$, it takes an expected $S$ more steps for those to collide, leaving a single common ancestor.  This suggests that when $k\approx 2S$, we expect the ancestry to collapse to a single node---one initial district will appear in every plan.  This will be further discussed below in  Table \ref{tab:PC}.
\end{remark}

\FloatBarrier

\subsection{Non-uniform weights}

Above we assumed that each active node chooses its parent uniformly at random.  In practice, this is not how the SMC code works; 
the weights  depend on  graph properties of the partitions.  The weight on node $j$ at level $i$ is given by $w_i^{(j)}=\frac{\left(\tau(G_i^{(j)})\right)^{\rho-1}}{\left|\partial G_i^{(j)}\right|}$, where the numerator is $\tau$, the product of the number of spanning trees  in the pieces of the partial plan $G_i^{(j)}$, raised to a power $\rho -1$, and the denominator is the size of the edge cut.  

\begin{figure}[bht!]
\centering
\includegraphics[width=2.4in]{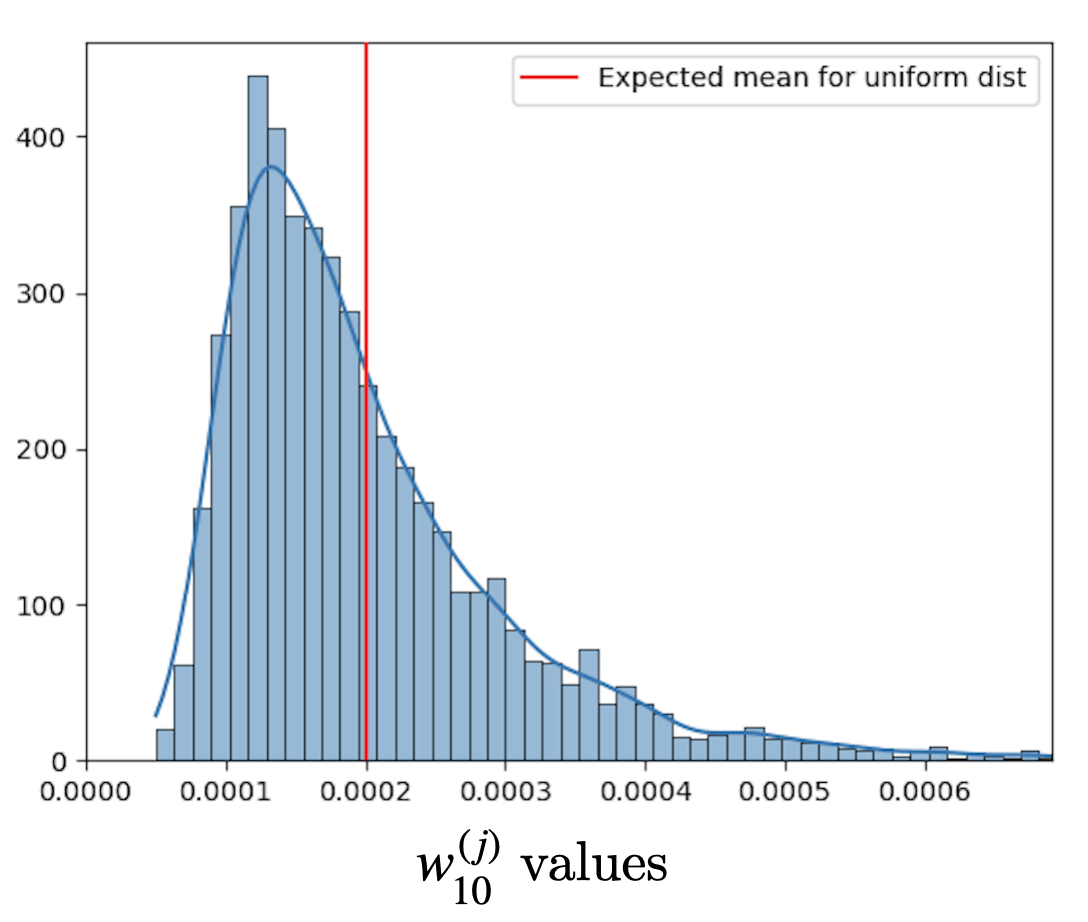}
\caption{Truncation of a long-tailed histogram of weights in a descendancy diagram on state Senate districts in New Mexico  ($k=42, i=10, S=5000$, default settings).  If weights were uniform, the distribution of weights would be concentrated at the red line.  Instead, when drawing the 33rd district in this SMC process, some 32-district partial plans are over 100 times likelier than others to be chosen.}\label{fig:NM-weights}
\end{figure}

When $\rho\neq 1$ these weight factors will give wildly different probability of selection to plans based on their compactness, due to $\tau$ values for districts that can easily differ by $10^{100}$ in realistic problems.\footnote{In particular, $\tau> e^{N}$ for many planar graphs on $N$ vertices (for instance $\tau\sim e^{1.6N}$ for triangular lattices), and Congressional districts might typically contain $N=500$ precincts. This is one reason that validating on a $6\times 6$ grid with $6$ districts is inadequate to see salient effects of scale: in that setting, $\tau$ values for individual districts can only differ by a factor of 15.  See related discussion in \citet[\S5.1]{DeFord2021Recombination}.}  Even at $\rho=1$, plans with longer boundaries will be weighted down. (See Figure~\ref{fig:NM-weights} for an empirical example with $\rho=1$.) These inter-generational weights thus present a compactness-related bias pulling away from uniformity for any choice of parameters.

\FloatBarrier

Next we show that non-uniform weights exacerbate the sample repetition.
Recall that $X_{i}$ was a random variable denoting the number of active nodes at level $i$, where $X_1 \equiv S$ and parents are chosen uniformly at random at each level.  
We now set up our second model by fixing some non-uniform distribution over $S$ nodes at each level $1, \dots, k-2$ for the selection of parents.  Fixing those distributions, we initialize $Y_1\equiv S$, and let $Y_{i}$ be a random variable denoting the number of active nodes at level $i$ with the specified parent selection probabilities. 

\begin{lemma}[Uniform descendancy minimizes ancestor collapse]\label{lem:unif} For $i \geq 2$, let $X_i$ give the count of active nodes at level $i$ randomized over uniform descendancy diagrams, while $Y_i$ is defined instead with the same non-uniform distribution on each level.  Then 
    $$\E[Y_{i} \mid Y_{i-1} = a] \leq \E[X_{i} \mid X_{i-1} = a]. $$
\end{lemma}

\begin{proof}
    %From above, we know $\e[X_{s,i} | X_{s,i-1} = a] = s \left(1-(1-1/s)^a)\right)$, so we aim to show the expression on the left hand side is less than this value. 
    Consider the random variable $Y_i$. For $j \in \{1,2, \ldots , S\}$, let $p_j$ denote the probability that an individual at level $i-1$ chooses $j$ as their parent in level $i$ (note the $p_j$'s may also vary with $i$). 
    We will need H\"older's Inequality, which states that for $p,q \in [1, \infty)$ satisfying $1/p + 1/q = 1$ and any two vectors $\mathbf{u}$ and $\mathbf{v}$, $\| \langle \mathbf{u} ,  \mathbf{v} \rangle \|_1 \leq  \| \mathbf{u}  \| _p  \| \mathbf{v} \| _q$. 
    We apply this to the vectors $\mathbf{u}$ where $u_j = (1-p_j)/(S-1)$ and $\mathbf{v} = (1,1,1, \dots , 1)$. Note $$ \| \langle \mathbf{u}, \mathbf{v} \rangle  \| _1 =  \|  \mathbf{u}  \| _1 = \sum_{j = 1}^S \frac{1-p_j}{S-1} = 1.$$ Using $p = a$ and $q = a/(a-1)$, we see that 
    \begin{align*}
      & \| \mathbf{u} \| _a  = \left( \sum_{j = 1}^S \left(  \frac{1-p_j}{S-1} \right)^a \right)^{1/a}  = \left( \frac{1}{(S-1)^a} \sum_{j = 1}^S (1-p_j)^a \right)^{1/a}  \\
      & \| \mathbf{v} \| _{a/(a-1)} = \left( \sum_{j = 1}^S 1^{a/(a-1)} \right)^{(a-1)/a} = S^{(a-1)/a} 
    \end{align*}
    Putting this together with H\"older's inequality, we see that 
$$\left( \frac{1}{(S-1)^a} \sum_{j = 1}^S (1-p_j)^a \right)^{1/a}   S^{(a-1)/a}  \geq 1  \qquad \hbox{\rm and} \qquad 
        \sum_{j = 1}^S (1-p_j)^a \geq \frac{(S-1)^a}{S^{a-1}} = S \left(1 - \frac{1}{S}\right)^a.$$
    We now see that 
    \begin{align*}
        \E[Y_{i} \mid Y_{i-1} = a] &= \sum_{j = 1}^S (1- (1-p_j)^a) = S - \sum_{j = 1}^S (1-p_j)^a \\ &\leq S - S \left(1 - \frac{1}{S}\right)^a = S \left(1-\left(1-\frac1s\right)^a\right) = \E[X_{i} \mid X_{i-1} = a].
    \end{align*}
    This completes the proof. 
\end{proof}

We find  empirically that the weights at each level frequently have a distribution shaped  like the one shown in Figure~\ref{fig:NM-weights}, which was drawn from a run on New Mexico state Senate districts.   In New Mexico, it was common to see max-to-median weight ratios of 10 within a generation, and max-to-min ratios of 30, even with $\rho=1$.  In New York's Senate districts, these ratios were commonly 30 and 2000, respectively.  
As we will see, skews of this kind will tend to significantly increase the collision rate.
%We model this with the simplified assumption that all nodes in a given run are choosing from the same non-uniform distribution over potential parents.  

\begin{table}[htb!]
    \centering
\begin{tikzpicture}

\node at (4,2.4) {{\em Lowest level with a "mega-ancestor" accounting for $\varphi$ share of the final outputs}};
\node at (4,2) {{\em (Low table entries for high $\varphi$ are signs of extreme redundancy)  }};
\node at (0,0)  {\begin{tabular}{|c||c|c|c|c|}
    \hline
      $F$ & $S=10$&100&1000&5000 \\
     \hline
     \hline
     $\varphi=.01$    & ---& ---& 3.8 & \hcell  15.4 \\
     \hline
     .1&    ---& 5.5&51.1 &\hcell  256.3 \\
     \hline
     .25&  2.5& 18.7&188.3 	&\hcell  957.8\\
     \hline
     .5 & 5.5& 60.2&622.2 	&\hcell  3187.3\\
     \hline
     .75 &14.6&144.7&1433.5& 	7092.4\\
     \hline 
     1& 17.9&201.9&2065.2& 	9966.2\\
     \hline
    \end{tabular}};
\node at (0,-1.9) {Uniform weights};
\node at (8,0)  {\begin{tabular}{|c||c|c|c|c|}
    \hline
     $F$ & $S=10$&100&1000&5000 \\
     \hline
     \hline
     $\varphi=.01$ &--- &--- &2.0&\hcell  2.0\\
     \hline
     .1&---&2.0&4.6&\hcell  78.4  \\
     \hline
     .25&2.0&2.0&19.7&\hcell  320.1\\
     \hline
     .5 &2.0&2.8&65.8&\hcell  1032.7\\
     \hline
     .75 &2.0&5.3&151.7&2401.9\\
     \hline 
     1&2.7&11.1&232.3& 3533.2\\
     \hline
    \end{tabular}};
\node at (8,-1.9) {$100:1:1\dots:1$ weights};

\node at (4,-4)
{\begin{tabular}{|c||c|c|}
\hline 
& NM & NY \\
$F$ & $k=42$ & $k=63$\\
& $S=5000$ & $S=5000$\\
\hline \hline 
$\varphi=.01$&\hcell  3.2 (100\%) &\hcell  2.9 (100\%) \\ \hline 
$.1$ &\hcell  8.8 (100\%) &\hcell  6.9 (100\%)\\ \hline 
$.25$ &\hcell  14.9 (71.4\%) &\hcell  11.9 (99\%) \\ \hline 
$.5$ &\hcell  17.0 (23.5\%) &\hcell  18.8 (74\%) \\ \hline 
\end{tabular}};
\node at (4,-5.9) {Actual runs of SMC code};

\end{tikzpicture}
\caption{$F(D,\varphi)$ reports the lowest level at which some node is an ancestor to $\varphi$ share of the bottom generation.  (This is vacuous if $\varphi S\le 1$.)
The tables show estimated expectations for $F(D,\varphi)$ with uniform weights and with stylized non-uniform weights. Each cell value is obtained by averaging over 1000 trials. 
Reading across the bottom row in the case of uniform weights (top left) confirms that $A(S,2S)\approx 1$.
Collapse is much worse with stylized non-uniform weights (top right). 
However, the actual runs on New Mexico and New York (bottom) show that the empirical repetition can be far greater than either prediction. For instance, the NM cell at $\varphi=.25$ tells us that 714 out of 1000 trials in New Mexico had a node serving as ancestor to 25\% of the final generation, and that node on average occurred at level 14.9 (meaning $25\%$ of the plans had $42-15=27$ districts in common).
}
    \label{tab:PC}
\end{table}

Recall that one goal (Question~\ref{Q:Gdj}) is to measure the distribution of statistics $G(D,j)$ defined as the number of final plans in $D$ with $j$ districts exactly in common.  
A mathematically more natural expression that contains the needed information but is a bit harder to phrase in English is $F(D,\varphi)$, the lowest level at which there is a "mega-ancestor" accounting for $\varphi$ share of the final generation:
$$F(D,\varphi):=\min\{i : \exists\ j ~\hbox{\rm with}~ d(i,j)> \varphi\cdot S\}.$$

\FloatBarrier

Table \ref{tab:PC} shows an simulated comparison of $F(D,\varphi)$ between the case of uniform weights and a simple non-uniform setup where one node at each layer is 100 times more likely to be selected than each of the others (that is, the weights are $100:1:1\dots:1$). Each are run 1000 times.  As we would expect given Lemma~\ref{lem:unif}, the ancestor collapse is accelerated in the non-uniform case.  
But even this significantly understates the actual repetition in 1000 runs to make SMC samples from New Mexico and New York; we can compare this as well to Tapp's expert affidavit finding roughly that $F(D,.6)<31$ for a $S=5000$ sample in New York.  The excess degeneracy in real-world samples is partly because the graph partition step itself can boost repetition; if partial progress has created a hard-to-split remainder, this creates yet another scenario in which a generation may be filled out with repeats.
(This is referred to as graph "bottlenecks" elsewhere in this note.)

%Runs of the SMC code on real-world cases reveal highly non-uniform weights, suggesting a much faster ancestor collapse than the values estimated in the last section. One example of this for a sample of New Mexico state Senate districts is shown in Figure \ref{fig:NM-weights}. {\color{magenta} need to do some tweaking here due to the trees vs. boundary length vs. rho vs. alpha issue} A key source of non-uniformity in SMC for redistricting is the size of the relative magnitudes of the number of spanning trees per district, as these are used as the weights at each step of the selection process. The range of these values can vary over enormous orders of magnitude (see ) on real world examples and while the uniform spanning tree measure concentrates mass on districts with a relatively large number of  trees the absolute differences in magnitude can still be significant and these propagate throughout the SMC process. 

\begin{figure}
    \centering
\begin{tikzpicture}
\node at (0,0) {\includegraphics[width=4.6in]{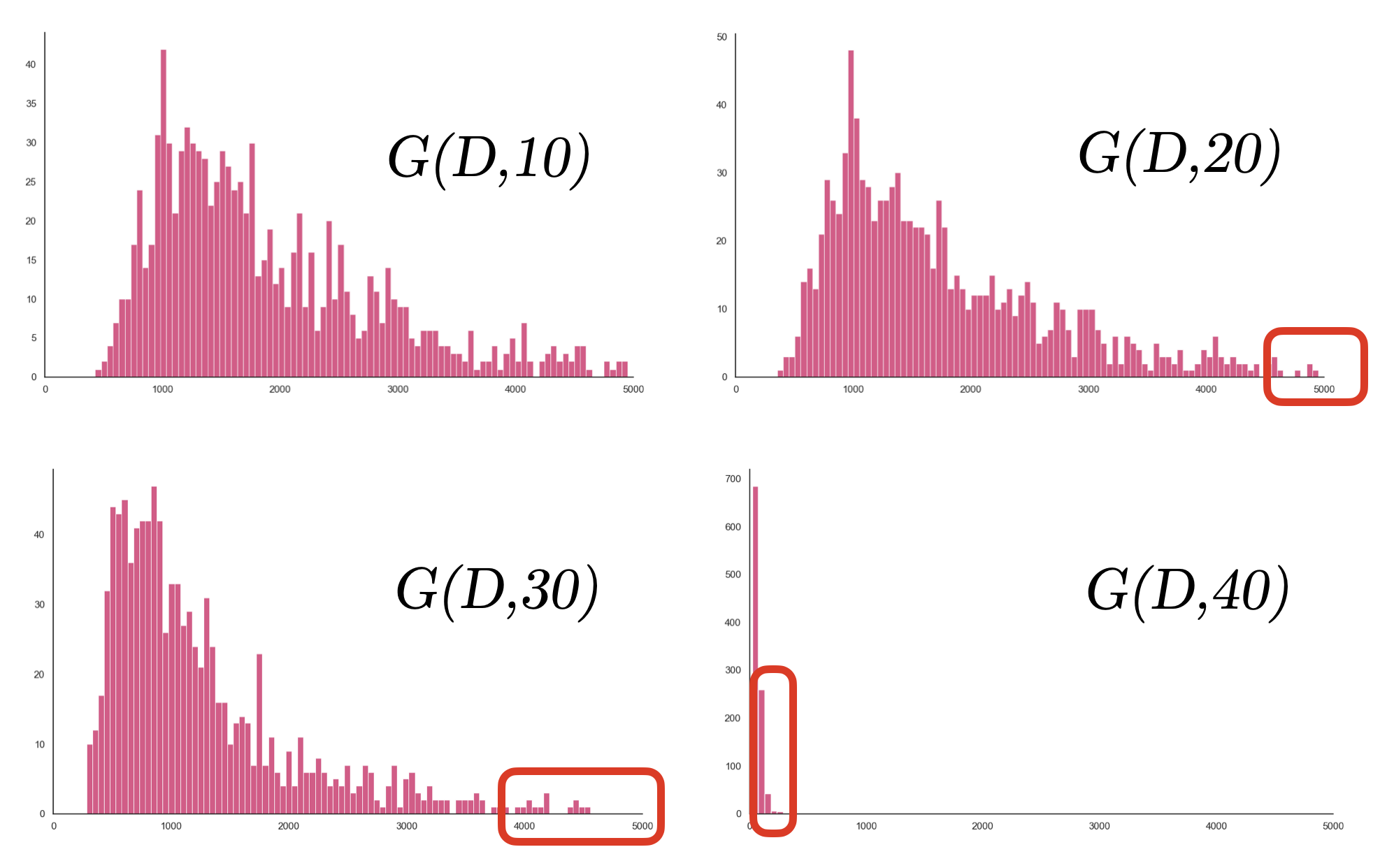}};
\node at (0,-4.2) {Observed $G(D,j)$ for New Mexico runs};

\end{tikzpicture}
\caption{We now recast the  information from the same 1000 SMC runs to spell out the redundancy.  The histograms show how many plans have a set of $j$ districts exactly in common.  A few observations are visible in the marked areas.  On several runs, more than 98\% of outputs shared 20 districts out of 42 exactly in common.  On several runs, more than 90\% of outputs shared 30 districts exactly in common.  And finally, more than 1/4 of the runs have 50 or more plans in their outputs that are nearly identical statewide, sharing 40 out of 42 districts.  These plots are from New Mexico; the effects in New York are still more extreme.}
    \label{fig:GDj}
\end{figure}

Figure~\ref{fig:GDj} turns this around and shows the distribution of $G(D,j)$, the number of plans with $j$ districts in common, for the same 1000 SMC runs on the New Mexico Senate.  
Though the redundancy is striking in New Mexico, it is even worse in New York, the other real-world example we track throughout this note.

\FloatBarrier
%%%%%%%%%%%%%
%%%%%%%%%%%%%
\section{Convergence guarantees and diagnostics}

In \cite{mccartan2023sequential}, the main convergence result is a weak central limit theorem stated as follows.

\begin{proposition}[McCartan--Imai Prop 4.2]\label{prop:smc-conv}
 Let $\pi_S = \sum_{j=1}^S w^{(j)} \delta_{[\xi(j)]}$ be the weighted particle approximation generated by [their SMC Algorithm]. Then for all measurable $h$ on unlabeled plans, as $S \to \infty$,
$$\sqrt{S} \left( 
\E_{\pi_S}\left[ h\left([\xi]\right)\right] - 
\E_{\pi}\left[ h\left([\xi]\right)\right]
\right) \xrightarrow{d} \mathcal{N}(0,V_{\rm SMC}(h))$$
for some asymptotic variance $V_{\rm SMC}(h)$.
\end{proposition}

As the authors note, this is convergence in probability rather than almost sure convergence. 
This style of convergence result does not rule out significant sample repetition.\footnote{Indeed, even stronger central limit theorems like those for Markov chains do not rule out this level of repetition, but observed repetition is far less severe for Markov chain methods than for SMC in the redistricting application---and, in any case, they admit far larger samples with current methods.}
To understand the guarantees better, consider the following construction.

\begin{example}[Controlled repetition sampler (CRS)]
Fix a parameter $0<\alpha<1/2$.  
Given a distribution $\pi$ on a state space $\Omega$, let the {\em controlled repetition sampler with parameter $\alpha$} be defined as follows:  a sample of size $S$ is constructed by taking one draw $x\sim \pi$ and adding $x$ to the sample $\lceil S^\alpha \rceil$ times, each with weight $1/S$.  
Next, create a smaller SMC sample of  $S'=S-\lceil S^\alpha \rceil$ plans.  If a plan in this smaller sample is assigned weight $w^{(j)}$ by SMC, we add it to our CRS sample with normalized weight  $w^{(j)}S'/S$.
\end{example}

Despite the amount of repetition present, where up to $\sqrt S$ plans in the sample of $S$ are fully identical, this CRS method still satisfies the conclusions of Proposition~\ref{prop:smc-conv}.  (This is proved in Appendix~\ref{app:CRS}.)  Note that the power $S^\alpha$ could be replaced with any function of $S$ such that $f(S)/\sqrt{S-f(S)}\to 0$.  

Of course, since $\sqrt S /S\to 0$, the first and simplest way to minimize the effects of repetition is to employ very large samples.
When large samples are computationally expensive, one may be tempted to combine multiple separate samples rather than enlarging a single sample.
Indeed, in their small validation example (dividing the $6\times 6$ grid into $k=6$ districts) the SMC authors have averaged 24 independent runs to obtain the estimates shown in \citet[Fig 4]{mccartan2023sequential} rather than enlarging the size of a single sample past $S=10,000$.  
In other published work \citep{McCartan_ALARM}, the same authors and collaborators present ensembles of 5000 maps for all 50 states, and do so in most cases by combining subsamples from multiple SMC runs.
But we know of no theory for this:  it is not clear what balance of the sample size $S$ and the number of separate runs would constitute best practices for users of SMC for redistricting.\footnote{For the illustrative example CRS, the ideal structure would be many samples of size 1.  For SMC, this would fail, because a large $S$ is already needed for the reweighting of the final sample to take the $J$ term in the target distribution into account at all.}

In SMC/\Redist, the main convergence diagnostic  is the Gelman--Rubin $\hat{R}$ statistic, which compares within-sample variance to between-sample variance, typically applied to batches of 2500 or 5000 plans.\footnote{Information on the $\hat{R}$ diagnostic can be found at \cite{Gelman1992Inference,Vehtari2021Rank}.}  
Discussing related issues for MCMC, Charles Geyer \citep{Geyer} wrote of several popular techniques for using many short runs in place of one long run as "worse than useless": they raise your confidence and produce cosmetically better samples while giving no reason to believe the outputs are close to the target.  For now, the computational costs of SMC/\Redist lock users in to a many-short-runs  framework.

%%%
\section{Discussion}

Since spanning tree methods were introduced in redistricting around 2018, many authors who need a single exemplar of a plan (such as to serve as the starting point for a Markov chain) obtain one by recursively partitioning a tree down to districts.  One way to view SMC for redistricting is that it constructs the whole sample by running this seeding process many times.  Employing the structure of a descendancy diagram allows the use of weights that help control the properties of the sample, but at the cost of introducing significant combinatorial redundancy.  

\subsection*{Repetition undermines statistical conclusions and data visualization.}
SMC is sometimes claimed to produce nearly independent samples from arbitrary distributions.\footnote{For instance, this was at one point explicit in the \Redist documentation at \url{https://perma.cc/YV37-JZNR}.} 
However, district repetition can create massive dependencies, with many plans being identical on large regions, for the sample sizes currently possible with SMC/\Redist.  Numerous factors that are present in real use cases---non-uniform weights stemming from compactness terms and restricted choices in the decision tree caused by the connection topology of the state, among others---can contribute to extreme repetition, which is visible in plots. 
A highly redundant sample does not allow for reliable outlier analysis because its percentile statistics will often be far from those of the target distribution---for instance, if 20\% of a sample had repeated or identical features, that would manifestly call into question its usefulness to identify 1\% outliers.
To put the same point in visual terms: significant repetition clearly undermines the use of an SMC ensemble to infer the shape of a distribution of summary statistics, such as in histograms and boxplots.  
If statistics such as $\hat{R}$ are used as diagnostics on a histogram, they should be applied to the height of each bar.  A box-and-whiskers plot on Pennsylvania, as shown in \citet[Fig 7]{mccartan2023sequential}, would need 90 $\hat{R}$ calculations (top whisker, third quartile, median, first quartile, bottom whisker for each district).\footnote{See \cite{colorado} for an example where a diagnostic metric, in that case the Kolmogorov--Smirnov statistic, is used in this way to support a box-and-whiskers plot.}

\subsection*{Final reweighting cannot compensate for small samples.} The SMC process leverages importance sampling by constructing an initial sample through a descendancy diagram and then reweighting according to the target distribution $\pi$ only when the sample is complete.  Before the final reweighting, the sample is approximately distributed by the spanning tree distribution $\tau$ on partitions---but somewhat distorted by repetition, labeling bias, and other artifacts of the construction.  The authors intend to use this method to target arbitrary distributions 
$\pi(\xi) \propto \displaystyle e^{-J(\xi)}\tau(\xi)^\rho$, but, in particular, the energy functional $J$ is never used until the descendancy diagram is complete; partitions that are never encountered by the tree-based generation process cannot be rescued by reweighting.  
Thus attempts to target a general $\pi$ with SMC will fare no better than applying one-shot reweighting to previously known methods to sample from $\tau$, such as MCMC.\footnote{If $J$ decomposes into a district-by-district score, then there is a way to take it into account in the intergenerational weights, but this is not done by default and it is not available for general $J$.}
The accuracy depends on emitting a large and diverse sample from the descendancy process.
With samples in the tens of thousands on practical problems,  many legally relevant events will never be observed---by contrast, current Markov chain methods for sampling from $\tau$ can be run to billions of (accepted) steps.  

\subsection*{All-purpose SMC ensembles are unsuited for novel measurements.} Courts have often expressed an interest in the presence of individual districts with particular properties.  For instance, in the litigation challenging the Pennsylvania Congressional plan, the court strongly discouraged the division of Pittsburgh across two districts (and the special master was said to treat it as a "disqualifying feature" of a plan).\footnote{{\em Carter v. Chapman} (2022), see \url{https://www.pacourts.us/assets/opinions/Supreme/out/J-20-2022mo.pdf}.}  
This preference emerged long after the initial expert work had been conducted.
In the New Mexico legal challenge, the parties to litigation debated whether it was disqualifying if any district "contains more than 60\% of the state’s active oil wells."\footnote{{\em Republican Party~vs.~Oliver} (2023), see \href{https://www.nmlegis.gov/Redistricting2021/LitigationDocuments}{\tt nmlegis.gov} links, particularly {\em Plaintiff's Opposed Motion to Exclude Expert Report \& Expert Testimony of Dr.~Jowei Chen} \citep{NM-doc}.}
These examples illustrate that it will frequently be legally relevant to know if some district-level property is common in a neutrally constructed ensemble.  If the most-repeated district in some Pennsylvania SMC run happens to divide Pittsburgh, say, the ensemble can give a highly misleading answer.  
Ensembles constructed with current methodology, including the ALARM ensembles published in {\em Scientific Data} \citep{McCartan_ALARM}, are unsuited for estimating the frequency with which district-level properties occur, beyond those properties tested at the time of data generation.\footnote{In \cite{McCartan_ALARM}, the  $\hat R$ statistic is used on a preset list of metrics:  "Finally, we evaluate the convergence of the algorithm for the specific set of summary statistics described above that are of interest to practitioners."  It is unclear exactly how $\hat{R}$ is being applied.}

\subsection*{All issues compound with large numbers of districts.} Large numbers of districts (large $k$) exacerbate all of the problems described here.\footnote{Even with small numbers of districts, there can be issues:  if few districts have many units, then the inter-generational weights can be highly non-uniform, which will likewise boost repetition.  We suspect this is a real but less severe worry.}   The published validation efforts only go up to $k\le 6$ districts, and the authors themselves have used subdivision to make the problem more tractable---the ALARM 50-state project breaks up Texas ($k=38$), Florida ($k=27$), and California ($k=52$) into three or more pieces, assembling statewide Congressional plans by combining smaller runs on separate regions.\footnote{See \href{https://github.com/alarm-redist/fifty-states/blob/main/analyses/TX_cd_2020/doc_TX_cd_2020.md}{Texas}, \href{https://github.com/alarm-redist/fifty-states/blob/main/analyses/FL_cd_2020/doc_FL_cd_2020.md}{Florida}, and \href{https://github.com/alarm-redist/fifty-states/blob/main/analyses/CA_cd_2020/doc_CA_cd_2020.md}{California} {\sc readme} files from the ALARM 50-state project.}  
Ensembles that have been modularized in this way have an unknown relationship to the target distribution on the full state.

\bigskip

SMC for redistricting is a highly valuable addition to the literature on sampling methods for graph partitioning.  In addition to the theoretical contribution, the \Redist  package that implements SMC is commendably well-documented and user-friendly,  and its authors have designed default settings that allow beginners to complete runs on many full-sized problems.  
Also notably, the main code faced by users of \Redist is in R, a programming language that is popular in social science graduate training.  This has created the conditions for even first-time users of graph algorithms to generate materials for expert reports in court cases across the United States, but without a deep understanding of diagnostics and limitations.\footnote{Experts indicating that SMC/\Redist was their first exposure to graph algorithms, and sometimes even to algorithmic sampling methods more broadly, filed reports not only in New Mexico and New York, but also North Carolina, Pennsylvania, Louisiana, and Texas, and possibly more.  As far as we encountered, they did not discuss convergence diagnostics or  warning messages about sample diversity in their expert reports or filings.}

In this note, we show that the combinatorics of descendancy diagrams combines with a host of other factors to create potentially severe district repetition in SMC samples.  In most cases at the full problem scale of a U.S. state, minimizing the effects of repetition would call for far larger sample sizes than is currently possible.
Issues compound when attempting to target distributions that differ from the spanning tree distribution $\tau$ or when the plans have more than about a dozen districts.
All of these observations counsel caution in using SMC on full-scale redistricting problems.

\bibliographystyle{imsart-nameyear}
\bibliography{GerryRefs}

\begin{thebibliography}{15}
% BibTex style file: imsart-nameyear.bst, 2017-11-03
% Default style options (sort=1,type=nameyear).
% Used options (sort=1,type=nameyear).

\bibitem[\protect\citeauthoryear{Bertoin}{2006}]{coagulations}
\begin{bbook}[author]
\bauthor{\bsnm{Bertoin},~\bfnm{Jean}\binits{J.}}
(\byear{2006}).
\btitle{Random Fragmentation and Coagulation Processes}.
\bpublisher{Cambridge University Press}.
\end{bbook}
\endbibitem

\bibitem[\protect\citeauthoryear{Cannon et~al.}{2022}]{cannon2022spanning}
\begin{bmisc}[author]
\bauthor{\bsnm{Cannon},~\bfnm{Sarah}\binits{S.}},
  \bauthor{\bsnm{Duchin},~\bfnm{Moon}\binits{M.}},
  \bauthor{\bsnm{Randall},~\bfnm{Dana}\binits{D.}} \AND
  \bauthor{\bsnm{Rule},~\bfnm{Parker}\binits{P.}}
(\byear{2022}).
\btitle{Spanning tree methods for sampling graph partitions}.
\end{bmisc}
\endbibitem

\bibitem[\protect\citeauthoryear{Clelland et~al.}{2021}]{colorado}
\begin{barticle}[author]
\bauthor{\bsnm{Clelland},~\bfnm{Jeanne}\binits{J.}},
  \bauthor{\bsnm{Colgate},~\bfnm{Haley}\binits{H.}},
  \bauthor{\bsnm{DeFord},~\bfnm{Daryl}\binits{D.}},
  \bauthor{\bsnm{Malmskog},~\bfnm{Beth}\binits{B.}} \AND
  \bauthor{\bsnm{Sancier-Barbosa},~\bfnm{Flavia}\binits{F.}}
(\byear{2021}).
\btitle{{Colorado in Context: Congressional Redistricting and Competing
  Fairness Criteria in Colorado}}.
\bjournal{Journal of Computational Social Science}
\bvolume{5}
\bpages{180-226}.
\end{barticle}
\endbibitem

\bibitem[\protect\citeauthoryear{{NM Legislative Defendants}}{2022}]{NM-doc}
\begin{bmisc}[author]
\bauthor{\bsnm{{NM Legislative Defendants}}}
(\byear{2022}).
\btitle{Legislative Defendants' Response to Plaintiffs' Proposed Motion to
  Exclude}.
\bnote{Fifth Judicial District of New Mexico. {\em Republican Party of New
  Mexico v. Maggie Toulouse Oliver} (2022).,
  \url{https://www.nmlegis.gov/Redistricting2021/Litigation\%20Docs/292\%20-\%20September\%2025,\%202023\%20Plaintiff's\%20Opposed\%20Motion\%20to\%20Exclude\%20Expert\%20Report\%20&\%20Expert\%20Testimony\%20of\%20Dr.\%20Jowei\%20Chen.pdf}}.
\end{bmisc}
\endbibitem

\bibitem[\protect\citeauthoryear{DeFord, Duchin and
  Solomon}{2021}]{DeFord2021Recombination}
\begin{barticle}[author]
\bauthor{\bsnm{DeFord},~\bfnm{Daryl}\binits{D.}},
  \bauthor{\bsnm{Duchin},~\bfnm{Moon}\binits{M.}} \AND
  \bauthor{\bsnm{Solomon},~\bfnm{Justin}\binits{J.}}
(\byear{2021}).
\btitle{Recombination: A {Family} of {Markov} {Chains} for {Redistricting}}.
\bjournal{Harvard Data Science Review}
\bvolume{3}.
\bnote{\url{https://hdsr.mitpress.mit.edu/pub/1ds8ptxu}}.
\end{barticle}
\endbibitem

\bibitem[\protect\citeauthoryear{Durrett}{2008}]{durrett2002geneticdrift}
\begin{bbook}[author]
\bauthor{\bsnm{Durrett},~\bfnm{Rick}\binits{R.}}
(\byear{2008}).
\btitle{Probability Models for DNA Sequence Evolution},
\bedition{Second} ed.
\bpublisher{Springer}.
\end{bbook}
\endbibitem

\bibitem[\protect\citeauthoryear{Gelman and Rubin}{1992}]{Gelman1992Inference}
\begin{barticle}[author]
\bauthor{\bsnm{Gelman},~\bfnm{Andrew}\binits{A.}} \AND
  \bauthor{\bsnm{Rubin},~\bfnm{Donald~B.}\binits{D.~B.}}
(\byear{1992}).
\btitle{Inference from Iterative Simulation Using Multiple Sequences}.
\bjournal{Statistical Science}
\bvolume{7}
\bpages{457--472}.
\end{barticle}
\endbibitem

\bibitem[\protect\citeauthoryear{Geyer}{2011}]{Geyer}
\begin{bincollection}[author]
\bauthor{\bsnm{Geyer},~\bfnm{Charles~J.}\binits{C.~J.}}
(\byear{2011}).
\btitle{Introduction to {M}arkov chain {M}onte {C}arlo}.
In \bbooktitle{Handbook of {M}arkov chain {M}onte {C}arlo}.
\bseries{Chapman \& Hall/CRC Handb. Mod. Stat. Methods}
\bpages{3--48}.
\bpublisher{CRC Press, Boca Raton, FL}.
\bmrnumber{2858443}
\end{bincollection}
\endbibitem

\bibitem[\protect\citeauthoryear{Kenny et~al.}{}]{Redist}
\begin{bunpublished}[author]
\bauthor{\bsnm{Kenny},~\bfnm{C.~T.}\binits{C.~T.}},
  \bauthor{\bsnm{McCartan},~\bfnm{C.}\binits{C.}},
  \bauthor{\bsnm{Fifield},~\bfnm{B.}\binits{B.}} \AND
  \bauthor{\bsnm{Imai},~\bfnm{K.}\binits{K.}}
\btitle{{\sf Redist}: Computational algorithms for redistricting simulation.}
\bnote{R Package. \url{https://CRAN.R-project.org/package=redist}}.
\end{bunpublished}
\endbibitem

\bibitem[\protect\citeauthoryear{Liu, Chen and Logvinenko}{2001}]{LCL}
\begin{bincollection}[author]
\bauthor{\bsnm{Liu},~\bfnm{Jun~S.}\binits{J.~S.}},
  \bauthor{\bsnm{Chen},~\bfnm{Rong}\binits{R.}} \AND
  \bauthor{\bsnm{Logvinenko},~\bfnm{Tanya}\binits{T.}}
(\byear{2001}).
\btitle{A theoretical framework for sequential importance sampling with
  resampling}.
In \bbooktitle{Sequential {M}onte {C}arlo methods in practice}.
\bseries{Stat. Eng. Inf. Sci.}
\bpages{225--246}.
\bpublisher{Springer, New York}.
\bmrnumber{1847794}
\end{bincollection}
\endbibitem

\bibitem[\protect\citeauthoryear{McCartan and
  Imai}{2023}]{mccartan2023sequential}
\begin{barticle}[author]
\bauthor{\bsnm{McCartan},~\bfnm{Cory}\binits{C.}} \AND
  \bauthor{\bsnm{Imai},~\bfnm{Kosuke}\binits{K.}}
(\byear{2023}).
\btitle{{Sequential Monte Carlo for sampling balanced and compact redistricting
  plans}}.
\bjournal{The Annals of Applied Statistics}
\bvolume{17}
\bpages{3300 -- 3323}.
\bdoi{10.1214/23-AOAS1763}
\end{barticle}
\endbibitem

\bibitem[\protect\citeauthoryear{McCartan et~al.}{2022}]{McCartan_ALARM}
\begin{barticle}[author]
\bauthor{\bsnm{McCartan},~\bfnm{Cory}\binits{C.}},
  \bauthor{\bsnm{Kenny},~\bfnm{Christopher~T.}\binits{C.~T.}},
  \bauthor{\bsnm{Simko},~\bfnm{Tyler}\binits{T.}},
  \bauthor{\bsnm{Garcia~III},~\bfnm{George}\binits{G.}},
  \bauthor{\bsnm{Wang},~\bfnm{Kevin}\binits{K.}},
  \bauthor{\bsnm{Wu},~\bfnm{Melissa}\binits{M.}},
  \bauthor{\bsnm{Kuriwaki},~\bfnm{Shiro}\binits{S.}} \AND
  \bauthor{\bsnm{Imai},~\bfnm{Kosuke}\binits{K.}}
(\byear{2022}).
\btitle{Simulated redistricting plans for the analysis and evaluation of
  redistricting in the {U}nited {S}tates}.
\bjournal{Scientific Data}
\bvolume{9}
\bpages{689}.
\bdoi{https://doi.org/10.1038/s41597-022-01808-2}
\end{barticle}
\endbibitem

\bibitem[\protect\citeauthoryear{Raab and Steger}{1998}]{balls_and_bins}
\begin{binproceedings}[author]
\bauthor{\bsnm{Raab},~\bfnm{Martin}\binits{M.}} \AND
  \bauthor{\bsnm{Steger},~\bfnm{Angelika}\binits{A.}}
(\byear{1998}).
\btitle{``{B}alls into {B}ins'' --- {A} Simple and Tight Analysis}.
In \bbooktitle{Randomization and Approximation Techniques in Computer Science}
(\beditor{\bfnm{Michael}\binits{M.}~\bsnm{Luby}}, \beditor{\bfnm{Jos{\'e}
  D.~P.}\binits{J.~D.~P.}~\bsnm{Rolim}} \AND
  \beditor{\bfnm{Maria}\binits{M.}~\bsnm{Serna}}, eds.)
\bpages{159--170}.
\bpublisher{Springer Berlin Heidelberg}, \baddress{Berlin, Heidelberg}.
\end{binproceedings}
\endbibitem

\bibitem[\protect\citeauthoryear{Tapp}{2022}]{Tapp}
\begin{bmisc}[author]
\bauthor{\bsnm{Tapp},~\bfnm{Kristopher}\binits{K.}}
(\byear{2022}).
\btitle{Second {A}ffidavit of {D}r. {K}ristopher {R}. {T}app, {P}h{D}}.
\bnote{Supreme Court of the State of New York. {\em Harkenrider v. Hochul}
  (2022). \url{https://perma.cc/29X3-59CX}}.
\end{bmisc}
\endbibitem

\bibitem[\protect\citeauthoryear{Vehtari et~al.}{June 2021}]{Vehtari2021Rank}
\begin{barticle}[author]
\bauthor{\bsnm{Vehtari},~\bfnm{Aki}\binits{A.}},
  \bauthor{\bsnm{Gelman},~\bfnm{Andrew}\binits{A.}},
  \bauthor{\bsnm{Simpson},~\bfnm{Daniel}\binits{D.}},
  \bauthor{\bsnm{Carpenter},~\bfnm{Bob}\binits{B.}} \AND
  \bauthor{\bsnm{Bürkner},~\bfnm{Paul-Christian}\binits{P.-C.}}
(\byear{June 2021}).
\btitle{Rank-Normalization, Folding, and Localization: An Improved $\hat{R}$
  for Assessing Convergence of {MCMC} (with Discussion)}.
\bjournal{Bayesian Analysis}
\bvolume{16}.
\end{barticle}
\endbibitem

\end{thebibliography}

\clearpage 

\appendix

\section{Weak CLT for controlled repetition sampler}
\label{app:CRS}

\begin{proposition}
For any $0<\alpha<1/2$, the controlled repetition sampler with parameter $\alpha$ satisfies 
$$\sqrt{S} \left( 
\E_{\pi_S}\left[ h\left([\xi]\right)\right] - 
\E_{\pi}\left[ h\left([\xi]\right)\right]
\right) \xrightarrow{d} \mathcal{N}(0,V_\alpha(h))$$ for any measurable function $h$, where $V_\alpha$ is an asymptotic variance.
\end{proposition}

\begin{proof}

We will use $\pi_{S'}$ to denote the distribution for the SMC sample of size $S'$ used above (before it is reweighted by $S'/S$), and $\pi_S$ to denote our complete sample of size $S$, where 
\[\pi_S = \frac{\lceil S^{1/3} \rceil}{S} \delta_{[\xi_0]} + \frac{S'}{S}\pi_{S'} =  \frac{\lceil S^{1/3} \rceil}{S} \delta_{[\xi_0]} + \sum_{j = 1}^{S'} \frac{S'}{S} w^{(j)} \delta_{[\xi(j)]}.\]
Let $h$ be any measurable function on unlabelled plans. 
We know, from Proposition~\ref{prop:smc-conv} applied to $\pi_{S'}$, that as $S \rightarrow \infty$ (which also implies $S' \rightarrow \infty)$, 
\[ \sqrt{S'} \left( 
\E_{\pi_{S'}}\left[ h\left([\xi]\right)\right] - 
\E_{\pi}\left[ h\left([\xi]\right)\right]
\right) \xrightarrow{d} \mathcal{N}(0,V_{\rm SMC}(h))\]
We define a random variables $Y'_{S} := \sqrt{S'} \left( 
\E_{\pi_{S'}}\left[ h\left([\xi]\right)\right] - 
\E_{\pi}\left[ h\left([\xi]\right)\right]
\right)$, and let $Z \sim \mathcal{N}(0, V_{SMC})$.  We know for any $a$, \[
\lim_{S \rightarrow \infty} \Pr(Y'_{S} \leq a) = \Pr(Z \leq a). 
\]
Because $Z$ is a continuous random variable and thus $\Pr(Z = a) = 0$, it's also true that for any sequence $b_{S}$ with $\lim_{S \rightarrow \infty} b_{S} = a$, 
\[
\lim_{S \rightarrow \infty} \Pr(Y'_{S} \leq b_{S}) = \Pr(Z \leq a). 
\]

Let $Y_S : = \sqrt{S} \left( 
\E_{\pi_{S}}\left[ h\left([\xi]\right)\right] - 
\E_{\pi}\left[ h\left([\xi]\right)\right] \right)$, and note that we are trying to show the sequence  $Y_S$ converges in distribution to $\mathcal{N}(0, V_{SMC})$. We see that 
\[
\E_{\pi_S} \left[ h\left([\xi]\right)\right] = \frac{\lceil S^{1/3} \rceil}{S} h([\xi_0]) + \frac{S'}{S} \E_{\pi_{S'}}\left[ h\left([\xi]\right)\right].
\]
Using that we can rewrite $\E_{\pi}\left[ h\left([\xi]\right)\right] = \frac{\lceil S^{1/3} \rceil}{S} \E_{\pi}\left[ h\left([\xi]\right)\right] + \frac{S'}{S} \E_{\pi}\left[ h\left([\xi]\right)\right]$,
 it follows that 
\begin{align*}
Y_S &= \sqrt{S} \left( 
 \frac{\lceil S^{1/3} \rceil}{S} h([\xi_0]) + \frac{S'}{S} \E_{\pi_{S'}}\left[ h\left([\xi]\right)\right] - 
\E_{\pi}\left[ h\left([\xi]\right)\right] \right)
\\&=  \frac{\lceil S^{1/3} \rceil}{\sqrt{S}} h([\xi_0]) 
- \frac{\lceil S^{1/3} \rceil}{\sqrt{S}}  \E_{\pi}\left[ h\left([\xi]\right)\right]
+ \frac{S'}{\sqrt{S}} \left( \E_{\pi_{S'}}\left[ h\left([\xi]\right)\right] 
- \E_{\pi}\left[ h\left([\xi]\right)\right] \right)
\\&= \frac{\lceil S^{1/3} \rceil}{\sqrt{S}} \left( h([\xi_0]) - \E_{\pi}\left[ h\left([\xi]\right)\right] \right)  
+ \sqrt{\frac{S'}{S}} \cdot \sqrt{S'} \left( \E_{\pi_{S'}}\left[ h\left([\xi]\right)\right] 
- \E_{\pi}\left[ h\left([\xi]\right)\right] \right)  
\\&= \frac{\lceil S^{1/3} \rceil}{\sqrt{S}} \left( h([\xi_0]) - \E_{\pi}\left[ h\left([\xi]\right)\right] \right)  + \sqrt{\frac{S'}{S}} Y'_{S}
\end{align*}

If we consider $\Pr(Y_S \leq a)$ for any $a$, we see that 
\begin{align*}
    \Pr\left(Y_S \leq a\right)) 
    &= \Pr\left( \frac{\lceil S^{1/3} \rceil}{\sqrt{S}} \left( h([\xi_0]) - \E_{\pi}\left[ h\left([\xi]\right)\right] \right)  + \sqrt{\frac{S'}{S}} Y'_{S'} \leq a
    \right)
    \\&= \Pr \left( Y'_{S} \leq \sqrt{\frac{S}{S'}} \left(a + \frac{\lceil S^{1/3} \rceil}{\sqrt{S}} \left(  \E_{\pi}\left[ h\left([\xi]\right)\right] - h([\xi_0]) \right) \right) \right)
\end{align*}
Note for $b_{S} := \frac{\sqrt{S}}{\sqrt{S'}} \left(a + \frac{\lceil S^{1/3} \rceil}{\sqrt{S}} \left(  \E_{\pi}\left[ h\left([\xi]\right)\right] - h([\xi_0]) \right)\right)$, we see that $\lim_{S \rightarrow \infty} b_{S} = a$, as $\sqrt{S/S'}$ goes to one and $\frac{\lceil S^{1/3} \rceil}{\sqrt{S'}}$ goes to 0. We conclude that for all $a$, 
\[
 \lim_{S \rightarrow \infty} \Pr(Y_S \leq a) =  \lim_{S \rightarrow \infty} \Pr(Y'_{S} \leq b_S) = \Pr(Z \leq a),
\]
and so $Y_S$ converges in distribution to $\mathcal{N}(0, V_{\alpha})$, as claimed.  Note this result would hold with $S^{1/3}$ replaced by $S^\alpha$ for any $\alpha < 1/2$.

\end{proof}

\end{document}